\documentclass[A4]{amsart}
\input xypic
\input xy \xyoption{all}
\usepackage{MnSymbol}

\usepackage{bbm}
\usepackage{amsmath}

\newtheorem{theorem}{Theorem}[section]
\newtheorem{proposition}[theorem]{Proposition}
\newtheorem{lemma}[theorem]{Lemma}

\theoremstyle{definition}
\newtheorem{definition}[theorem]{Definition}
\newtheorem{bigremark}[theorem]{Remark}

\newtheorem{example}[theorem]{Example}

\newcounter{bean}

\newcommand{\seqm}[3]{\ensuremath{#1\stackrel{#2}
 {\longrightarrow}#3}}
\newcommand{\seqmm}[5]{\ensuremath{#1\stackrel{#2}
 {\longrightarrow}#3\stackrel{#4}{\longrightarrow}#5}}
\newcommand{\seqmmm}[7]{\ensuremath{#1\stackrel{#2}
 {\longrightarrow}#3\stackrel{#4}{\longrightarrow}#5
  \stackrel{#6}{\longrightarrow}#7}}

\newcommand{\floor}[1]{\ensuremath{\left\lfloor #1 \right\rfloor}}

\newcommand{\paren}[1]{\ensuremath{\left( #1 \right)}}
\newcommand{\br}[1]{\ensuremath{\left\{ #1 \right\}}}

\newcommand{\sbr}[1]{\ensuremath{\left[#1\right]}}

\newcommand{\cvee}[3]{\displaystyle\bigvee^{#2}_{#1}#3}

\newcommand{\cplus}[3]{\displaystyle\bigoplus^{#2}_{#1}#3}

\newcommand{\cunion}[3]{\displaystyle\bigcup^{#2}_{#1}#3}

\newcommand{\cunionmulti}[4]{\displaystyle\bigcup^{#3}
_{\renewcommand{\arraystretch}{0.6}\begin{matrix}\scriptstyle #1 \cr \scriptstyle #2\end{matrix}\renewcommand{\arraystretch}{1.2}}#4}

\newcommand{\cset}[2]{\br{#1\,\,\middle\vert\,\,#2}}

\newcommand{\qqed}{\hfill\square}

\renewcommand{\k}{\mathbf k}

\newcommand{\mc}[1]{\ensuremath{\mathcal{#1}}}
\newcommand{\mb}[1]{\ensuremath{\mathbb{#1}}}
\newcommand{\mr}[1]{\ensuremath{\mathring{#1}}}

\newcommand{\ID}{\ensuremath{\mathbbm{1}}}

\newcommand{\bd}{\ensuremath{\partial}}

\newcommand{\wcolon}{\ensuremath{\,\colon\,}}

\DeclareMathOperator{\Int}{int}
\DeclareMathOperator{\cl}{cl}

\newcommand{\z}[1]{\mathcal Z_{#1}}
\newcommand{\zk}{\mathcal Z_K}
\newcommand{\hz}[1]{\widehat{\mathcal Z}_{#1}}
\newcommand{\hzk}{\widehat{\mathcal Z}_K}
\newcommand{\rz}[1]{\mathbb R\mathcal Z_{#1}}
\newcommand{\rzk}{\mathbb R\mathcal Z_K}
\newcommand{\hrz}[1]{\mathbb R\widehat{\mathcal Z}_{#1}}
\newcommand{\hrzk}{\mathbb R\widehat{\mathcal Z}_K}

\begin{document}
\title{$\frac{n}{3}$-Neighbourly Moment-Angle Complexes and their Unstable Splittings}

\author{Piotr Beben}
\address{\scriptsize{School of Mathematics, University of Southampton,Southampton SO17 1BJ, United Kingdom}} 
\email{P.D.Beben@soton.ac.uk} 
\author{Jelena Grbi\'c} 
\address{\scriptsize{School of Mathematics, University of Southampton,Southampton SO17 1BJ, United Kingdom}}  
\email{J.Grbic@soton.ac.uk} 

\subjclass[2010]{Primary 55P15, 55U10, 13F55}
\keywords{polyhedral product, moment-angle complex, toric topology, Stanley-Reisner ring, Golod ring, neighbourly simplicial complexes} 

\begin{abstract}
Given an $\frac{n}{3}$-neighbourly simplicial complex $K$ on vertex set $[n]$,
we show that the moment-angle complex $\zk$ is a $co$-$H$-space if and only if $K$ satisfies a homotopy analogue of the Golod property.
\end{abstract}
\maketitle

\section{Introduction}

A ring $R=\k[v_1,\ldots,v_n]/I$ for $I$ a homogeneous ideal is said to be \emph{Golod} 
if the multiplication and higher Massey products are trivial in $\mathrm{Tor}^+_{\k[v_1,\ldots,v_n]}(R,\k)$.
The Poincar\'e series of the ring $\mathrm{Tor}_{R}(\k,\k)$ of $R$ represents a rational function whenever $R$ is Golod~\cite{MR0138667}. 
In combinatorics, one says a simplicial complex $K$ on vertex set $[n]=\{1,\ldots,n\}$ is \emph{Golod} over $\k$ 
if its Stanley-Reisner ring $\k[K]$ is Golod. 
Equivalently, when $\k$ to be a field or $\mb Z$, $K$ is Golod over $\k$ if all cup products and Massey products 
(of positive degree elements) are trivial in the cohomology ring $H^*(\zk;\k)$ of the \emph{moment-angle complex} $\zk$.
This follows from isomorphisms of graded commutative algebras given in~\cite{MR1897064,MR2255969,MR2117435,MR0441987}
\begin{equation}
\label{EHochster}
H^*(\zk;\k)\cong \mathrm{Tor}_{\k[v_1,\ldots,v_n]}(\k[K],\k) \cong \cplus{I\subseteq[n]}{}{\tilde H^*(\Sigma^{|I|+1}|K_I|;\k)}
\end{equation}
where $K_I$ is the restriction of $K$ to vertex set $I\subseteq[n]$. 
The multiplication on the right hand side is induced by maps
$\iota_{I,J}\wcolon\seqm{|K_{I\cup J}|}{}{|K_I\ast K_J|\cong |K_I|\ast |K_J|\simeq\Sigma |K_I|\wedge |K_J|}$
that realize the canonical inclusions \seqm{K_{I\cup J}}{}{K_I\ast K_J} whenever $I$ and $J$ are non-empty and disjoint, 
and are defined zero otherwise.  
Thus, the Golod condition is equivalent to $\iota_{I,J}$ inducing trivial maps on $\k$-cohomology for all disjoint non-empty $I,J\subsetneq[n]$,
as well as Massey products vanishing.
 
Since the cohomology ring $H^*(\zk)$ is in its simplest algebraic form when $K$ is Golod, 
it is natural to ask what the homotopy type of $\zk$ is in this case. 
Quite a lot of work has been done in this direction 
~\cite{MR2138475,MR2321037,MR3084441,arXiv:1306.6221,arXiv:1211.0873,BebenGrbic1,IK1},
and it has been conjectured that $K$ is Golod if and only if $\zk$ is a co-$H$-space.
One can then determine the homotopy type of $\zk$ since, by~\cite{arXiv:1306.6221,MR2673742}, 
$\zk$ is a $co$-$H$-space if and only if $\zk\simeq\bigvee_{I\subseteq[n]}\Sigma^{|I|+1}|K_I|$.
This conjecture has been verified for flag complexes~\cite{arXiv:1211.0873,MR2344344}, 
and rationally or localized at large primes~\cite{BerglundRational,BebenGrbic1}.

A simplicial complex $K$ on vertex set $[n]$ is said to be $k$-\emph{neighbourly} if every subset of $k$ or less vertices in $[n]$ is a face of $K$.
That is, $K_I$ is a simplex for each $I\subseteq[n]$, $|I|\leq k$. If $K$ is $k$-neighbourly then it is $k'$-neighbourly for all $k'\leq k$.
When $K$ is $\frac{n}{3}$-neighbourly, we show that a homotopy analogue of Golodness is equivalent to the moment-angle complex $\zk$ being a $co$-$H$-space.

\begin{theorem}
\label{TMain}
Let $K$ be $\frac{n}{3}$-neighbourly. 
Then $\zk$ is a $co$-$H$-space (equivalently, $\zk\simeq \cvee{I\subseteq [n]}{}{\Sigma^{|I|+1}|K_I|}$) if and only if 
$$
\Sigma^{|I\cup J|}\iota_{I,J}\colon\seqm{\Sigma^{|I\cup J|}|K_{I\cup J}|}{}{\Sigma^{|I\cup J|}|K_I\ast K_J|}
$$
is nullhomotopic for each disjoint non-empty $I,J\subsetneq[n]$.
\end{theorem}

\begin{bigremark}
The necessity part of Theorem~\ref{TMain} is true in general
(see Proposition~\ref{PNecessary} below, or Proposition~$5.4$ in~\cite{BebenGrbic1}).
Sufficiency generally holds up to a certain coherence condition~(\cite{BebenGrbic1}, Theorem~$1.2$). 
Since this coherence condition involves joins $|K_{I_1}\ast\cdots\ast K_{I_m}|$ over disjoint non-empty 
$I_1,\ldots,I_m\subsetneq[n]$ and $m\geq 3$, coherence should be given to us for free when $K$ is $\frac{n}{3}$-neighbourly, 
as $3$-or-more-fold joins are contractible in this case. This is the heuristic behind Theorem~\ref{TMain}.
\end{bigremark}

\begin{example}
Suppose $K$ is $\frac{n}{2}$-neighbourly. Then each $|K_I|$ is contractible when $|I|\leq \frac{n}{2}$,
is at least $(\floor{\frac{n}{2}}-1)$-connected, and has dimension at most $n-2$ when $K$ is not a simplex.
From isomorphism~\eqref{EHochster}, or by Theorem~$6.33$~\cite{MR1897064}, 
one sees $\zk$ is contractible when $K$ is a simplex or has the homotopy type of an $n$-connected $CW$-complex of at most dimension $2n-1$.
Thus $\zk$ is a $co$-$H$-space by the Freudenthal suspension theorem. 
This also follows from distinct arguments found in~\cite{BebenGrbic1,IK1}. 
On the other hand, it is an easy application of Theorem~\ref{TMain}.
Since for each disjoint non-empty $I,J\subsetneq[n]$ at least one of $K_I$ or $K_J$ must be a simplex,
each $\seqm{|K_{I\cup J}|}{\iota_{I,J}}{|K_I\ast K_J|}$ must be nullhomotopic since $|K_I\ast K_J|\simeq \Sigma |K_I|\wedge |K_J|$ is contractible,
so $\zk$ is a $co$-$H$-space by Theorem~\ref{TMain}. 
\end{example}

The first examples outside the stable range occur when $K$ is $(\frac{n}{2}-1)$-neighbourly.

\begin{example}
Suppose $K$ is $(\frac{n}{2}-1)$-neighbourly and $n$ is even.
Then either (1) $K=\bd\Delta^{\frac{n}{2}-1}\ast\bd\Delta^{\frac{n}{2}-1}$ and $\zk$ is a sphere product;
(2) $K$ is $\bd\Delta^{\frac{n}{2}-1}\ast\bd\Delta^{\frac{n}{2}-1}$ with some faces removed;    
or else (3) $K$ has only one minimal missing face with $\frac{n}{2}$ vertices.
In the third case, each $\iota_{I,J}$ for disjoint $I$ and $J$ is nullhomotopic since at least one of $|K_I|$ or $|K_J|$ is contractible.
In the second case, $\iota_{I,J}$ is nullhomotopic since $|K|$ is at most $(n-4)$-dimensional and 
$|K_I\ast K_J|$ is at least $(n-4)$-connected. 
Then in the last two cases, $\zk$ is a $co$-$H$-space by Theorem~\ref{TMain}.
\end{example}

Massey products do not factor into Theorem~\ref{TMain}. In view of the isomorphism~\eqref{EHochster} only cup products are relevant. 
This is as might be expected when localized at the rationals, 
since $k$-connected $CW$-complexes of dimension $\leq 3k+1$ are rationally formal by~\cite{MR684544}.
Therefore they have trivial Massey products in their rational cohomology.

\section{Some Background}

\subsection{Conventions}

We take $-1$ to be the basepoint of $D^1=[-1,1]$.
The suspension $\Sigma X$ of a space $X$ is taken to be the reduced suspension 
$D^1\times X/(\{-1,1\}\times X\cup D^1\times\{\ast\})$ whenever $X$ is basepointed with basepoint $\ast$. 
Otherwise it is the unreduced suspension, in other words, the quotient space of $D^1\times X$ under identifications
$(-1,x)\sim \ast_{-1}$ and $(1,x)\sim \ast_{1}$. 
In any case, $\Sigma X$ is always basepointed, in the unreduced case the basepoint taken to be $\ast_{-1}$.

\subsection{Moment-Angle Complexes}

Let $K$ be a simplicial complex on $n$ vertices, and $(X,A)$ a $CW$-pair. 
The polyhedral product $(X,A)^K$ is the subspace of $X^{\times n}$ defined as the union
$$
(X,A)^K=\cunion{\sigma\in K}{}{Y^\sigma_1\times\cdots\times Y^\sigma_n},
$$
where $Y^\sigma_i=X$ if $i\in\sigma$, or else $Y^\sigma_i=A$ if $i\nin\sigma$.
The \emph{moment-angle complex} is $\zk=(D^2,S^1)^K$ and the \emph{real moment-angle complex} is $\rzk=(D^1,S^0)^K$,
where $S^1=\bd D^2$ and $S^0=\{-1,1\}=\bd D^1$.

Given $I\subseteq [n]$, we can think of $\z{K_I}$ and $\rz{K_I}$ as the subspace of $\zk$ and $\rzk$ 
consisting of those points whose coordinates not indexed by $I$ are basepoints. 
In this case we define $\zk^\ell$ and $\rzk^\ell$ to be the unions of $\z{K_I}$ and $\rz{K_I}$ respectively 
over all $I\subseteq [n]$ such that $|I|=\ell$, 
i.e. the subspaces of $\zk$ and $\rzk$ consisting of those points that have at least $n-\ell$ coordinates the basepoint.
The \emph{quotiented moment-angle complexes} are defined as
$
\hzk=\zk/\zk^{n-1}, 
$
$
\hrzk=\rzk/\rzk^{n-1},
$ 
$
\hz{K_I}=\z{K_I}/\z{K_I}^{|I|-1},
$ 
$
\hrz{K_I}=\rz{K_I}/\rz{K_I}^{|I|-1}.
$
Alternatively, they are the images of $\zk$, $\rzk$, $\z{K_I}$, and $\rz{K_I}$ under the quotient maps
\seqm{(D^2)^{\times n}}{}{(D^2)^{\wedge n}}, \seqm{(D^1)^{\times n}}{}{(D^1)^{\wedge n}}, 
\seqm{(D^2)^{\times |I|}}{}{(D^2)^{\wedge |I|}}, and \seqm{(D^1)^{\times |I|}}{}{(D^1)^{\wedge |I|}},
respectively. Then $\hrzk=\bigcup_{\sigma\in K} Y^\sigma_1\wedge\cdots\wedge Y^\sigma_n$, 
where $Y^\sigma_i=D^1$ if $i\in\sigma$, or $Y^\sigma_i=S^0$ if $i\nin\sigma$.

We say that a sequence $(I_1,\ldots,I_m)$ of subsets of a set $I$ is an \emph{(ordered) partition} of $I$ if they are mutually disjoint, 
non-empty, and $I_1\cup\cdots\cup I_m=I$. 
Given a partition $(I_1,\ldots,I_m)$ of $I$ with $|I|=k$ and $I_j=\{i_{j1},\ldots,i_{jk_j}\}$ in increasing order, 
by rearranging coordinates, we can, and usually will, think of the smash product $\hrz{K_{I_1}}\wedge\cdots\wedge\hrz{K_{I_m}}$ 
as the subspace of $(D^1)^{\wedge k}$ given by
\begin{equation}
\label{ERearranged}
\cset{(x_1,\ldots,x_k)\in (D^1)^{\wedge k}}
{(x_{i_{j1}},\ldots,x_{i_{jk_j}})\in\hrz{K_{I_j}}\mbox{ for each }j}
\end{equation}
(the order of the smash product therefore being irrelevant).
Then $\hrz{K_{I_1\cup\cdots\cup I_m}}$ is a subspace of $\hrz{K_{I_1}}\wedge\cdots\wedge\hrz{K_{I_m}}$,  
and we let 
$$
\hat\iota_{I_1,\ldots,I_m}\wcolon\seqm{\hrz{K_{I_1\cup\cdots\cup I_m}}}{}{\hrz{K_{I_1}}\wedge\cdots\wedge\hrz{K_{I_m}}} 
$$
denote the inclusion. This carries over similarly for moment-angle complexes $\zk$. 
By rearranging coordinates there is a homeomorphism $\hz{K_I}\cong \Sigma^{|I|}\hrz{K_I}$,
and so we see that the inclusion
$$
\hat\iota^1_{I_1,\ldots,I_m}\wcolon\seqm{\hz{K_{I_1\cup\cdots\cup I_m}}}{}{\hz{K_{I_1}}\wedge\cdots\wedge\hz{K_{I_m}}}
$$
is homeomorphic to the $k$-fold suspension of $\hat\iota_{I_1,\ldots,I_m}$ above. 

The homotopy type of a quotiented moment angle complex $\hrzk$ is not mysterious,
at least not with regards to the combinatorics of $K$.
The existence of a homeomorphism $\hrz{K_I}\cong \Sigma |K_I|$ was given in~\cite{MR2673742,BebenGrbic1},
and this is obtainable as a restriction of a homeomorphism 
$\Sigma|\Delta^{k-1}|\cong (D^1)^{\wedge k}$ where $|I|=k$. 
More generally, $\Sigma|\Delta^{k-1}|\cong (D^1)^{\wedge k}$ restricts to a homeomorphism
$$
\Sigma |K_{I_1}\ast\cdots\ast K_{I_m}|\cong \hrz{K_{I_1}}\wedge\cdots\wedge \hrz{K_{I_m}}
$$
where $(I_1,\ldots,I_m)$ is a partition of $I$. 
Consequently, $\hat\iota_{I_1,\ldots,I_m}$ is homeomorphic to the suspended inclusion 
$\Sigma\iota_{I_1,\ldots,I_m}\colon\seqm{\Sigma |K_{I_1\cup\cdots\cup I_m}|}{}{\Sigma|K_{I_1}\ast\cdots\ast K_{I_m}|}$.

\subsection{A Few Splitting Conditions}

If $I$ is a subset of $[n]$ with $k=|I|$ elements, and $\mc R=(s_1,\ldots,s_k)$ is any sequence of real numbers, 
let $i_\ell$ denote the $\ell^{th}$ smallest element in $I$, $s'_j$ denote the $j^{th}$ smallest element in the set $S=\{s_1,\ldots,s_k\}$, 
and $m=|S|$ be the number of distinct elements in $\mc R$.
Then assign to $\mc R$ the ordered partition $I_{\mc R}=(I_1,\ldots,I_m)$ of $I$ where $I_j=\cset{i_\ell\in I}{s_\ell=s'_j}$. 
For example, if $I=[4]$ and $\mc R=(-1,\pi,-1,0)$, then $[4]_{\mc R}=(\{1,3\},\{4\},\{2\})$.
Take the diagonal
$$
\vartriangle_n=\cset{(x_1,\ldots,x_n)\in\mb R^n}{x_1=\cdots=x_n}
$$
and consider the following subspace of the smash product $\mc P_n=(\mb R^n-\vartriangle_n)\wedge\Sigma|\Delta^{n-1}|$:
$$
\mc Q_K = 
\cunionmulti{y\in(\mb R^n-\vartriangle_n)}{(I_1,\ldots,I_m)=[n]_{y}}{}{\{y\}\wedge\Sigma |K_{I_1}\ast\cdots\ast K_{I_m}|},
$$
where $[n]_y$ is the ordered partition of $[n]$ for the given sequence $y$, as defined above.
Here we took $\mb R^n-\vartriangle_n$ to be without a basepoint, 
so $\mc P_n$ is the half-smash 
$((\mb R^n-\vartriangle_n)\times\Sigma|\Delta^{n-1}|)/(\mb R^n-\vartriangle_n)\times\{\ast\}$,
and is itself basepointed, as is $\mc Q_K$. 
Given any $\omega=(t_1,\ldots,t_{n-1},t,z)\in \Sigma^n |K|$ ($z\in |K|$ and $t_1,\ldots t_{n-1},t\in[-1,1]$), 
let $\beta=\max\{|t_1|,\ldots,|t_{n-1}|,0\}$, and define the map 
$$
\Phi_K\wcolon\seqm{\Sigma^n |K|}{}{\Sigma\mc Q_K}
$$
by
$$
\Phi_K(\omega)=
\begin{cases}
\ast & \mbox{if }\beta\in\{0,1\};\\
\paren{2\beta-1,(t_1,\ldots,t_{n-1},0),(t,z)} & \mbox{if }0<\beta<1.
\end{cases}
$$
\begin{remark}
Since $(t_1,\ldots,t_{n-1},0)$ cannot be on the diagonal $\vartriangle_n$,
one cannot generally define a nullhomotopy of $\Phi_K$ by trying to shift the parameter $\beta$ to $1$.
\end{remark}

\begin{definition}
$K$ is \emph{weakly coherently homotopy Golod} if $K$ is a single vertex, 
or (recursively) the vertex deletion $K\backslash\{i\}$ is weakly coherently homotopy Golod for each $i\in [n]$, 
and the map $\Phi_K\wcolon\seqm{\Sigma^n |K|}{}{\Sigma\mc Q_K}$ is nullhomotopic. 
\end{definition} 

\begin{bigremark}
Equivalently, 
$K$ is \emph{weakly coherently homotopy Golod} if $\Phi_{K_I}\wcolon\seqm{\Sigma^{|I|} |K_I|}{}{\Sigma\mc Q_{K_I}}$
is nullhomotopic for each $I\subseteq [n]$.
\end{bigremark}

\begin{proposition}[Theorem~$4.3$ in~\cite{BebenGrbic1}]
\label{PCoherence}
If $K$ is weakly coherently homotopy Golod, then $\zk$ is a co-$H$-space.~$\qqed$
\end{proposition}

On the other hand, a somewhat weaker looking condition is a necessary one.

\begin{proposition}[Proposition~$4.4$ in~\cite{BebenGrbic1}]
\label{PNecessary}
If $\zk$ is a co-$H$-space, then each of the suspended inclusions 
$$
\Sigma^{|I\cup J|}\iota_{I,J}\wcolon\seqm{\Sigma^{|I\cup J|} |K_{I\cup J}|}{}{\Sigma^{|I\cup J|} |K_I\ast K_J|}
$$
is nullhomotopic for every disjoint non-empty $I,J\subsetneq [n]$.
\end{proposition}

We include the proof since it is short and illustrative.

\begin{proof}[Proof of Proposition~\ref{PNecessary}]
Recall that $Y$ is a co-$H$-space with comultiplication \seqm{Y}{\psi}{Y\vee Y} if and only if 
the diagonal map \seqm{Y}{\vartriangle}{Y\times Y} is homotopic to the composite
\seqmm{Y}{\psi}{Y\vee Y}{include}{Y\times Y}. In particular, this implies that the reduced diagonal map
$\bar\vartriangle\wcolon\seqmmm{Y}{\vartriangle}{Y\times Y}{}{(Y\times Y)/(Y\vee Y)}{\cong}{Y\wedge Y}$
is nullhomotopic whenever $Y$ is a co-$H$-space.

Suppose $\zk$ is a co-$H$-space. Since $\z{K_{I\cup J}}$ is a retract of $\zk$, 
then $\z{K_{I\cup J}}$ is a co-$H$-space for any disjoint non-empty $I,J\subsetneq [n]$, 
and so \seqm{\z{K_{I\cup J}}}{\bar\vartriangle}{\z{K_{I\cup J}}\wedge \z{K_{I\cup J}}} is nullhomotopic. 
Then the composite 
$$
\seqmmm{\z{K_{I\cup J}}}{\bar\vartriangle}{\z{K_{I\cup J}}\wedge\z{K_{I\cup J}}}{}{\z{K_I}\wedge\z{K_J}}{}{\hz{K_I}\wedge \hz{K_J}}
$$
is nullhomotopic (where the second last map is the smash of the coordinate-wise projection maps onto $\z{K_I}$ and $\z{K_J}$, 
and the last map is the smash of quotient maps), and since it is equal to
$$
\seqmm{\z{K_{I\cup J}}}{quotient}{\hz{K_{I\cup J}}}{\hat\iota^1_{I,J}}{\hz{K_I}\wedge \hz{K_J}},
$$
this last composite is nullhomotopic as well. 
But $\z{K_{I\cup J}}$ is a co-$H$-space, so by Proposition~$2.11$ in~\cite{BebenGrbic1}, 
the quotient map \seqm{\z{K_{I\cup J}}}{}{\hz{K_{I\cup J}}} has a right homotopy inverse, implying $\hat\iota^1_{I,J}$ is nullhomotopic. 
Since $\hat\iota^1_{I,J}$ is up to homeomorphism the $(|I\cup J|)$-fold suspension of 
$\hat\iota_{I,J}\wcolon\seqm{\hrz{K_{I\cup J}}}{}{\hrz{K_I}\wedge \hrz{K_J}}$,
and since $\hat\iota_{I,J}$ is $\Sigma\iota_{I,J}$ up to homeomorphism, then $\Sigma^{|I\cup J|+1}\iota_{I,J}$ is nullhomotopic.

Now $\Sigma^{|I\cup J|}|K_{I\cup J}|$ is at most $(2(|I|+|J|)-1)$-dimensional and
$\Sigma^{|I\cup J|}|K_I\ast K_J|\simeq \Sigma^{|I\cup J|+1}|K_I|\wedge|K_J|$ is at least $(|I|+|J|)$-connected,   
so the suspension homomorphism
$$
\seqm{\sbr{\,\Sigma^{|I\cup J|}|K_{I\cup J}|\,,\,\Sigma^{|I\cup J|}|K_I\ast K_J|\,}}{}
{\sbr{\,\Sigma^{|I\cup J|+1}|K_{I\cup J}|\,,\,\Sigma^{|I\cup J|+1}|K_I\ast K_J|\,}} 
$$
is an isomorphism by the Freudenthal suspension theorem (\cite{MR1793722}, Theorem~$4.10$).
Thus $\Sigma^{|I\cup J|}\iota_{I,J}$ is nullhomotopic since $\Sigma^{|I\cup J|+1}\iota_{I,J}$ is.

\end{proof}

\begin{bigremark}

Proposition~\ref{PCoherence} is a consequence of the following.
Given any map $f\colon\seqm{\Sigma X}{}{Y}$, $f$ factors as \seqmm{\Sigma X}{\Sigma f^{ad}}{\Sigma\Omega Y}{ev}{Y} 
where $f^{ad}$ is the adjoint of $f$ and $ev$ the adjoint of the identity $\ID_{\Omega Y}$.
Taking $f$ to be an attaching map \seqm{\Sigma^n|K|}{}{\zk^{n-1}} whose cofiber is $\zk$,
it was shown in~\cite{BebenGrbic1} that $\Sigma f^{ad}$ factors through \seqm{\Sigma^n|K|}{\Phi_K}{\Sigma\mc Q_K}.
Then Proposition~\ref{PCoherence} follows from the fact that $f$ is nullhomotopic if and only if $\Sigma f^{ad}$ is.
\end{bigremark}

\subsection{A Reformulation of $\mc Q_K$}

It will be convenient to redefine $\mc Q_K$ and $\Phi_K$ in terms of quotiented moment-angle complexes.
Take the following subspace of $(\mb R^n-\vartriangle_n)\wedge (D^1)^{\wedge n}$: 
$$
\mc Q'_K \cong 
\cunionmulti{y\in(\mb R^n-\vartriangle_n)}{(I_1,\ldots,I_m)=[n]_{y}}{}{\{y\}\wedge\hrz{K_{I_1}}\wedge\cdots\wedge\hrz{K_{I_m}}}.
$$
As remarked before, there is a homeomorphism $\Sigma|\Delta^{n-1}|\cong (D^1)^{\wedge n}$ 
restricting for each partition $(I_1,\ldots,I_m)$ of $[n]$ to a homeomorphism
$\Sigma |K_{I_1}\ast\cdots\ast K_{I_m}|\cong \hrz{K_{I_1}}\wedge\cdots\wedge \hrz{K_{I_m}}$, 
such the inclusion $\hat\iota_{I_1,\ldots,I_m}$ is up to homeomorphism the suspended inclusion 
$\Sigma\iota_{I_1,\ldots,I_m}$. Thus, $\hrzk\cong \Sigma |K|$ and 
there is a homeomorphism $(\mb R^n-\vartriangle_n)\wedge \Sigma|\Delta^{n-1}|\cong(\mb R^n-\vartriangle_n)\wedge (D^1)^{\wedge n}$
restricting to a homeomorphism
$$
\mc Q_K \cong \mc Q'_K,
$$
with respect to which $\Phi_K$ is the map 
$$
\Phi'_K\wcolon\seqm{\Sigma^{n-1}\hrzk}{}{\Sigma\mc Q'_K}
$$
defined by: 
$\Phi'_K(t_1,\ldots,t_{n-1},x)=\paren{2\beta-1,(t_1,\ldots,t_{n-1},0),x}$ if $\beta=\max\{|t_1|,\ldots,|t_{n-1}|,0\}$ is between $0$ and $1$,
otherwise $\Phi'_K$ maps to the basepoint. 
We will work with $\mc Q'_K$ and $\Phi'_K$ instead of $\mc Q_K$ and $\Phi_K$ from now on.

\subsection{Clusters and Comultiplication}

Given integers $m\leq n$, a metric space $M$, a fixed point $x_n\in M$, and a function $f\colon\seqm{M^{\times(n-1)}}{}{\mb R}$,  
consider the configuration space $\mc A\subseteq M^{\times (n-1)}$ consisting of all points $y=(x_1\ldots x_{n-1})\in M^{\times (n-1)}$ 
such that each $x_i$ (including $x_n$) has a subset of $m$ of the $x_j$'s ($j\neq i$) clustered around it within distance $f(y)$.
We describe the topology of $\mc A$ when $M$ is the interval $(0,1)$, $m=\floor{\frac{n}{3}}$, $x_n=0$,
and $f(y)$ is proportional to the maximal distance between the $x_i$'s.

As a convention, $\min\emptyset=0$ and $\max\emptyset=0$.
If $y=(t_1,\ldots,t_{n-1})\in\mb R^{n-1}$, let $\mr y=(t_1,\ldots,t_{n-1},0)\in\mb R^n$.
Given $I\subseteq [n]$, $n\geq 2$, and $z=(t_1,\ldots,t_n)\in\mb R^n$, define the following non-negative values
$$
\nu^{z}_{I,i}=\min\cset{\max\cset{|t_i-t_j|}{j\in S}}{S\subseteq I-\{i\}\mbox{ and }|S|=\floor{\frac{n}{3}}},
$$
$$
\nu^{z}_I=\max\cset{\nu^{z}_{I,i}}{i\in I},
$$
$$
\delta^{z}=\frac{1}{n}(\max\{t_1,\ldots,t_n\}-\min\{t_1,\ldots,t_n\}).
$$
In other words, $\nu^{z}_{I,i}$ is the smallest value such that the interval $[t_i-\nu^{z}_{I,i},t_i+\nu^{z}_{I,i}]$ 
contains at least $\floor{\frac{n}{3}}$ of the $t_j$'s over $j\in I-\{i\}$, and is zero if no such value exists.
Consider the subspace $\mc A$ of the interior $\Int(D^{n-1})$ of the $(n-1)$-disk $D^{n-1}=[-1,1]^{\times (n-1)}$ given by
$$
\mc A = \cset{y\in \Int(D^{n-1})}{\nu^{\mr y}_{[n]}<\delta^{\mr y}}.
$$
Take the subset of ordered partitions of $[n]$
$$
\mc O = \cset{(I,J)}{\,I\cup J=[n],\,I\cap J=\emptyset,\,\mbox{and }|I|,|J|>\frac{n}{3}},
$$
and for each $(I,J)\in\mc O$, consider the subspace of $\Int(D^{n-1})$  
$$
\mc A_{I,J}=\cset{y=(t_1,\ldots,t_{n-1})\in\Int(D^{n-1})}
{t_j-t_i>\delta^{\mr y},\,\nu^{\mr y}_{I,i}<\delta^{\mr y},\,\nu^{\mr y}_{J,j}<\delta^{\mr y}\mbox{ for }i\in I,\,j\in J,\mbox{ where }t_n=0}.
$$
Let
$$
\hat{\mc A}_{I,J}=\cl(\mc A_{I,J})/\partial \mc A_{I,J}
$$
where $\cl(\mc A_{I,J})$ is the closure of $\mc A_{I,J}$, and $\partial \mc A_{I,J}$ is its boundary.

\begin{lemma}
\label{LComult}
$(i)$~The $\mc A_{I,J}$'s are mutually disjoint, 
and $\mc A$ is a union of $\mc A_{I,J}$ over all partitions $(I,J)\in\mc O$; and 
$(ii)$~each $\mc A_{I,J}$ is an open subspace of $\Int(D^{n-1})$ homeomorphic to the interior of $D^{n-1}$.
Hence 
$$
D^{n-1}/(D^{n-1}-\mc A)
\quad\cong\quad \cvee{(I,J)\in\mc O}{}{\hat{\mc A}_{I,J}}
\quad\cong\quad \cvee{(I,J)\in\mc O}{}{S^{n-1}}
$$
where $\hat{\mc A}_{I,J}\cong S^{n-1}$ for each $(I,J)\in\mc O$.
\end{lemma}

\begin{proof}[Proof of $(i)$]

Take any distinct $(I,J),(I',J')\in\mc O$, and let $t_n=0$ and $y=(t_1,\ldots,t_{n-1})\in\mb R^{n-1}$.
Since we have $t_i<t_j$ when $y\in\mc A_{I,J}$, $i\in I$ and $j\in J$; or when $y\in\mc A_{I',J'}$, $i\in I'$ and $j\in J'$;  
then $\mc A_{I,J}\cap\mc A_{I',J'}=\emptyset$ whenever $I\cap J'\neq\emptyset$ and $J\cap I'\neq\emptyset$.
Suppose $I\cap J'=\emptyset$ or $J\cap I'=\emptyset$. Since $(I,J)$ and $(I',J')$ are distinct partitions of $[n]$, 
then one of $I'\subsetneq I$, $I\subsetneq I'$, $J'\subsetneq J$, or $J\subsetneq J'$ holds. 
Without loss of generality, suppose $I'\subsetneq I$. 
If $y\in\mc A_{I,J}$, we have $\nu^{\mr y}_{I,i}<\delta^{\mr y}$ for any $i\in I$, so given that $i\in I-I'\subseteq J'$, 
there must be an $S\subseteq I$ such that $|S|=\floor{\frac{n}{3}}$ and $\max\cset{|t_i-t_j|}{j\in S}<\delta^{\mr y}$.  
Since $|I|,|J|,|I'|>\frac{n}{3}$ and $I'\subsetneq I$, then $|I-I'|<\frac{n}{3}$, so $S$ must contain at least one element $i'$ in $I'$.
But if $y\in\mc A_{I',J'}$, then $\max\cset{|t_i-t_j|}{j\in S}\geq |t_i-t_{i'}|\geq\delta^{\mr y}$ since $i\in J'$ and $i'\in I'$, a contradiction.
Thus $\mc A_{I,J}\cap\mc A_{I',J'}=\emptyset$ here as well. In summary, $\mc A_{I,J}$ and $\mc A_{I',J'}$ are disjoint for any distinct $(I,J),(I',J')\in\mc O$.

Suppose that $y\in\mc A$. 
Let $a=\min\{t_1,\ldots,t_n\}$ and $b=\max\{t_1,\ldots,t_n\}$. Since there are at most $n-2$ $t_i's$ in the interval $(a,b)$, 
there is an $l$ and an $r$ in $[n]$ such that there are no $t_i$'s in the interval $(t_l,t_r)$,
and $t_r-t_l>\frac{1}{n}(b-a)=\delta^{\mr y}$.
Then we have disjoint non-empty sets $I=\cset{i\in [n]}{t_i\leq t_l}$ and $J=\cset{j\in [n]}{t_j\geq t_r}$ partitioning $[n]$
with $t_j-t_i>\delta^{\mr y}$ for every $i\in I$ and $j\in J$. 
Notice that this last fact implies for every $k\in I$ there cannot be an $S\subseteq [n]-\{k\}$ such that $S$ 
contains an element in $J$ and: $(\ast)$ $|S|=\floor{\frac{n}{3}}$ and $\max\cset{|t_k-t_j|}{j\in S}<\delta^{\mr y}$. 
But since $\nu^{\mr y}_{[n]}<\delta^{\mr y}$ (as $y\in\mc A$), 
there is for each $k\in I$ at least one such $S\subseteq [n]-\{k\}$ such that $(\ast)$ holds, 
and we must have $S\subseteq I$ since $S$ cannot contain elements in $J$.  
Therefore $|I|>\frac{n}{3}$ and $\nu^{\mr y}_{I,k}<\delta^{\mr y}$ for every $k\in I$. 
Similarly, $|J|>\frac{n}{3}$ and $\nu^{\mr y}_{J,k}<\delta^{\mr y}$ for every $k\in J$.
It follows that $(I,J)$ is in $\mc O$, and $y\in\mc A_{I,J}$. Thus $\mc A\subseteq\bigcup_{(I,J)\in\mc O}\mc A_{I,J}$.

On the other hand, if $y\in A_{I,J}$, since either $\nu^{\mr y}_{I,i}<\delta^{\mr y}$ or $\nu^{\mr y}_{J,i}<\delta^{\mr y}$ for any $i\in[n]$, 
depending on whether $i\in I$ or $i\in J$, then $\nu^{\mr y}_{[n],i}<\delta^{\mr y}$ for each $i$, implying that $\nu^{\mr y}_{[n]}<\delta^{\mr y}$. 
Therefore $y\in\mc A$. Thus $\bigcup_{(I,J)\in\mc O}\mc A_{I,J}\subseteq\mc A$.

\vspace{.4cm} 
  
\noindent{\it Proof of $(ii)$.}
Notice that $\mc A$ is an open subspace of $\Int(D^{n-1})$ since it does not contain any points in its boundary 
$\cset{y\in \Int(D^{n-1})}{\nu^{\mr y}_{[n]}=\delta^{\mr y}}$. Thus $\mc A_{I,J}$ is also open by part~$(i)$. 

Given $(I,J)\in\mc O$, consider the point $b=(b_1,\ldots,b_{n-1})$ in $\mc A_{I,J}$ whose coordinates are defined 
for indices $i\in I$ and $j\in J$ by either:
$(1)$ $b_i=-\frac{1}{2}$ and $b_j=0$ if $n\in J$; or else $(2)$ $b_i=0$ and $b_j=\frac{1}{2}$ if $n\in I$.
Given $y\in\mc A_{I,J}$, let $y_t=(1-t)y+tb$. 
Notice that $\delta^{\mr y_t}=(1-t)\delta^{\mr y}+t(\frac{1}{2n})$, $\nu^{\mr y_t}_{I,i}=(1-t)\nu^{\mr y}_{I,i}$, 
$\nu^{\mr y_t}_{J,j}=(1-t)\nu^{\mr y}_{J,j}$ and $((1-t)t_j+tb_j)-((1-t)t_i+tb_i)=(1-t)(t_j-t_i)+t(\frac{1}{2})$ for any $i\in I$ and $j\in J$,
where $t_n=0$. It follows that $y_t$ is in $\mc A_{I,J}$ for each $t\in[0,1]$ whenever $y$ is,
and $\mc A_{I,J}$ deformation retracts onto the point $b$ via the homotopy sending $y\mapsto y_t$ at each time $t$.
Moreover, this is a linear homotopy with $y_t$ lying on the line through $b$ and $y$,
and $\mc A_{I,J}$ is an open subspace of $\Int(D^{n-1})$. We conclude that $\mc A_{I,J}$ is homeomorphic to $\Int(D^{n-1})$. 
 
\end{proof}

Since $\mc A$ is in the interior of $D^{n-1}$, $\bd D^{n-1}\subseteq D^{n-1}-\mc A$.
Of use to us will be the comultiplication
$$
\psi\wcolon\seqmm{S^{n-1}\cong D^{n-1}/\bd D^{n-1}}{}{D^{n-1}/(D^{n-1}-\mc A)}{\cong}
{\cvee{(I,J)\in\mc O}{}{}{S^{n-1}}}
$$
the first map being the quotient map, 
the last map the wedge of homeomorphisms \seqm{\hat{\mc A}_{I,J}}{}{S^{n-1}} 
restricting away from the basepoint to the homeomorphism 
$$
h\wcolon\seqm{\mc A_{I,J}}{}{\Int(D^{n-1})}
$$ 
from Lemma~\ref{LComult}.
On $S^{n-1}\wedge\hrzk=\Sigma^{n-1}\hrzk$, this gives the comultiplication 
$$
\psi_K=\psi\wedge \ID\wcolon\seqm{\Sigma^{n-1}\hrzk}{}
{\cvee{(I,J)\in\mc O}{}{\Sigma^{n-1}\hrzk}}
$$
where $\ID\colon\seqm{\hrzk}{}{\hrzk}$ is the identity.

\section{Proof of Theorem~\ref{TMain}}

Suppose $K$ is $\frac{n}{3}$-neighbourly. 
Notice that $\rz{K_{I'}}=(D^1)^{\times |I'|}$ whenever $|I'|\leq\frac{n}{3}$, in which case we have $\hrz{K_{I'}}=(D^1)^{\wedge |I'|}$. 
Thinking of $\hrz{K_I}\wedge\hrz{K_J}$ as the subspace of $(D^1)^{\wedge n}$ with coordinates rearranged as in \eqref{ERearranged}, 
we can define for any $(I,J)\in\mc O$ a map
$$
\lambda_{I,J}\wcolon\seqm{\Sigma^{n-1}\hrz{K_I}\wedge\hrz{K_J}}{}{\Sigma\mc Q'_K}
$$
that is given for any $\omega=(t_1,\ldots,t_{n-1},(x_1,\ldots,x_n))$ by
$$
\lambda_{I,J}(\omega)=
\begin{cases}
\ast & \mbox{if }\beta\in\{0,1\};\\
(2\beta-1,(t'_1,\ldots,t'_{n-1},0),(\varphi_1(x_1),\ldots,\varphi_n(x_n))) & \mbox{if }0<\beta<1,
\end{cases}
$$
where $\beta=\max\{|t_1|,\ldots,|t_{n-1}|,0\}$,
$y'=(t'_1,\ldots,t'_{n-1})=h^{-1}(t_1,\ldots,t_{n-1})$ for
$h^{-1}\colon\seqm{\Int(D^{n-1})}{\cong}{\mc A_{I,J}}$ the inverse of the homeomorphism $h$ in the definition of $\psi_K$,
and
$$
\varphi_i(x_i)=(1+x_i)\paren{\max\br{\,0\,,\,\frac{\delta^{\mr y'}-\nu^{\mr y'}_{[n],i}\,}{\delta^{\mr y'}}}}-1.
$$
The $t'_i$'s cannot all be $0=t'_n$ since by definition of $\mc A_{I,J}$, we have $t'_i<t'_j$ when $i\in I$ and $j\in J$, 
thus $\lambda_{I,J}(\omega)$ is in $\Sigma(\mb R^n-\vartriangle_n)\wedge (D^1)^{\wedge n}$ at the very least. 
By this last fact, and since we have $\varphi_i(x_i)=x_i$ for any $i$ such that there is an $I'\subseteq [n]$ 
with $i\in I'$, $|I'|>\frac{n}{3}$, and $t'_i=t'_j$ for any $j\in I'$, it follows that $\lambda_{I,J}(\omega)\in\Sigma\mc Q'_K$.
This is necessary since $\hrz{K_{I'}}$ is not necessarily $(D^1)^{\wedge |I'|}$ when $|I'|>\frac{n}{3}$. 
Otherwise, when $|I'|\leq\frac{n}{3}$, we have $\hrz{K_{I'}}=(D^1)^{\wedge |I'|}$, 
meaning $\varphi_i(x_i)$ is free to take on any value in $D^1=[-1,1]$.

\begin{lemma}
\label{LHomotopy}
Suppose $K$ is $\frac{n}{3}$-neighbourly, $n\geq 2$. Then $\Phi'_K\colon\seqm{\Sigma^{n-1}\hrzk}{}{\Sigma\mc Q'_K}$ is homotopic to the composite
$$
\zeta\,\wcolon\,\seqmmm{\Sigma^{n-1}\hrzk}{\psi_K}{\cvee{(I,J)\in \mc O}{}{\Sigma^{n-1}\hrzk}}{}{\cvee{(I,J)\in \mc O}{}{\Sigma^{n-1}\hrz{K_I}\wedge\hrz{K_J}}}{}{\Sigma\mc Q'_K}
$$
where the second map is the wedge of the suspended inclusions \seqm{\Sigma^{n-1}\hrzk}{\Sigma^{n-1}\hat\iota_{I,J}}{\Sigma^{n-1}\hrz{K_I}\wedge\hrz{K_J}} over $(I,J)\in \mc O$,
and the last map restricts to $\lambda_{I,J}$ on a summand $\Sigma^{n-1}\hrz{K_I}\wedge\hrz{K_J}$.
\end{lemma}

\begin{proof}

Define the homotopy $H_t\colon\seqm{\Sigma^{n-1}\hrzk}{}{\Sigma\mc Q'_K}$ for each 
$\omega=(t_1,\ldots,t_{n-1},(x_1,\ldots,x_n))$ in $\Sigma^{n-1}\hrzk$, and at each time $t\in[0,1]$, 
by mapping to the basepoint if $\beta=\max\{|t_1|,\ldots,|t_{n-1}|,0\}$ is $0$ or $1$, otherwise
$$
H_t(\omega)=(2\beta-1,(t_1,\ldots,t_{n-1},0),(\varphi_{1,t}(x_1),\ldots,\varphi_{n,t}(x_n))),
$$
where $\varphi_{i,t}(x_i)=(1-t)x_i+t\varphi_i(x_i)$ and $\varphi_i$ is as defined before for $y=(t_1,\ldots,t_{n-1})$ in place of $y'$.
We see that $H_t(\omega)\in \Sigma(\mb R^n-\vartriangle_n)\wedge (D^1)^{\wedge n}$ is in $\Sigma\mc Q'_K$ 
for similar reasons as outlined for $\lambda_{I,J}$.

Clearly $H_0=\Phi'_K$. 
By definition of $\mc A$, $\psi_K$, and $\varphi_i$, 
we have $\psi_K(\omega)=\ast$ and $\varphi_{i,1}(x_i)=\varphi_i(x_i)=-1$ (the basepoint in $D^1$) 
for some $i$ when $y\nin\mc A$, thus $\zeta(\omega)=H_1(\omega)=\ast$ in this case.
On the other hand $\mc A$ is the union of disjoint sets $\bigcup_{(I,J)\in\mc O}\mc A_{I,J}$, 
and if $y\in \mc A_{I,J}$ for some $(I,J)\in\mc O$, then by definition of $\psi_K$,
$\psi_K(\omega)=(h(y),(x_1,\ldots,x_n))$ in the summand $\Sigma^{n-1}\hrzk$ corresponding to $(I,J)$,
in which case 
\begin{align*}
\zeta(\omega)&=\lambda_{I,J}\circ\Sigma^{n-1}\hat\iota_{I,J}(\psi_K(\omega))\\ & = \lambda_{I,J}(h(y),(x_1,\ldots,x_n))\\
& = (2\beta-1,(h^{-1}(h(y)),0),(\varphi_1(x_1),\ldots,\varphi_n(x_n)))\\
& = (2\beta-1,(t_1,\ldots,t_{n-1},0),(\varphi_1(x_1),\ldots,\varphi_n(x_n)))\\
& = H_1(\omega).
\end{align*}
Thus, $H_1$ is the composite $\zeta$.

\end{proof}

\begin{proof}[Proof of Theorem~\ref{TMain}] 
Necessity follows from Proposition~\ref{PNecessary}.
The theorem is trivial to check when $n=1$ since $\zk=D^2$ here, so assume $n\geq 2$.
Since $K$ is $\frac{n}{3}$-neighbourly, then so is $K_I$, or $\frac{|I|}{3}$-neighbourly in particular.
Since the maps $\Phi'_K$ and $\hat\iota_{I,J}\colon\seqm{\hrzk}{}{\hrz{K_I}\wedge\hrz{K_J}}$ are homeomorphic to 
$\Phi_K$ and $\Sigma\iota_{I,J}\colon\seqm{\Sigma|K|}{}{\Sigma |K_I\ast K_J|}$ respectively,
sufficiency follows from Proposition~\ref{PCoherence} and applying Lemma~\ref{LHomotopy} for each $K_I$, 
$|I|\geq 2$, in place of $K$.
\end{proof}

\bibliographystyle{amsplain}
\bibliography{mybibliography}

\providecommand{\bysame}{\leavevmode\hbox to3em{\hrulefill}\thinspace}
\providecommand{\MR}{\relax\ifhmode\unskip\space\fi MR }
\providecommand{\MRhref}[2]{%
  \href{http://www.ams.org/mathscinet-getitem?mr=#1}{#2}
}
\providecommand{\href}[2]{#2}
\begin{thebibliography}{10}

\bibitem{MR2673742}
A.~Bahri, M.~Bendersky, F.~R. Cohen, and S.~Gitler, \emph{The polyhedral
  product functor: a method of decomposition for moment-angle complexes,
  arrangements and related spaces}, Adv. Math. \textbf{225} (2010), no.~3,
  1634--1668. \MR{2673742 (2012b:13053)}

\bibitem{MR2117435}
I.~V. Baskakov, V.~M. Bukhshtaber, and T.~E. Panov, \emph{Algebras of cellular
  cochains, and torus actions}, Uspekhi Mat. Nauk \textbf{59} (2004),
  no.~3(357), 159--160. \MR{2117435 (2006b:57045)}

\bibitem{BebenGrbic1}
Piotr Beben and Jelena Grbi\'c, \emph{Configuration spaces and polyhedral
  products}, preprint, arxiv:1409.4462.

\bibitem{BerglundRational}
Alexander Berglund, \emph{Homotopy invariants of davis-januszkiewicz spaces and
  moment-angle complexes}.

\bibitem{MR2344344}
Alexander Berglund and Michael J{\"o}llenbeck, \emph{On the {G}olod property of
  {S}tanley-{R}eisner rings}, J. Algebra \textbf{315} (2007), no.~1, 249--273.
  \MR{2344344 (2008f:13036)}

\bibitem{MR1897064}
Victor~M. Buchstaber and Taras~E. Panov, \emph{Torus actions and their
  applications in topology and combinatorics}, University Lecture Series,
  vol.~24, American Mathematical Society, Providence, RI, 2002. \MR{1897064
  (2003e:57039)}

\bibitem{MR2255969}
M.~Franz, \emph{The integral cohomology of toric manifolds}, Tr. Mat. Inst.
  Steklova \textbf{252} (2006), no.~Geom. Topol., Diskret. Geom. i Teor.
  Mnozh., 61--70. \MR{2255969 (2007f:14050)}

\bibitem{MR0138667}
E.~S. Golod, \emph{Homologies of some local rings}, Dokl. Akad. Nauk SSSR
  \textbf{144} (1962), 479--482. \MR{0138667 (25 \#2110)}

\bibitem{arXiv:1211.0873}
Jelena Grbi{\'c}, Taras Panov, Stephen Theriault, and Jie Wu, \emph{Homotopy
  types of moment-angle complexes for flag complexes}, preprint,
  arXiv:1211.0873, to appear in Trans. Amer. Math. Soc.
  http://dx.doi.org/10.1090/tran/6578.

\bibitem{MR2138475}
Jelena Grbi{\'c} and Stephen Theriault, \emph{Homotopy type of the complement
  of a configuration of coordinate subspaces of codimension two}, Uspekhi Mat.
  Nauk \textbf{59} (2004), no.~6(360), 203--204. \MR{2138475 (2005k:55023)}

\bibitem{MR2321037}
\bysame, \emph{The homotopy type of the complement of a coordinate subspace
  arrangement}, Topology \textbf{46} (2007), no.~4, 357--396. \MR{2321037
  (2008j:13051)}

\bibitem{MR3084441}
\bysame, \emph{The homotopy type of the polyhedral product for shifted
  complexes}, Adv. Math. \textbf{245} (2013), 690--715. \MR{3084441}

\bibitem{MR0441987}
Melvin Hochster, \emph{Cohen-{M}acaulay rings, combinatorics, and simplicial
  complexes}, Ring theory, {II} ({P}roc. {S}econd {C}onf., {U}niv. {O}klahoma,
  {N}orman, {O}kla., 1975), Dekker, New York, 1977, pp.~171--223. Lecture Notes
  in Pure and Appl. Math., Vol. 26. \MR{0441987 (56 \#376)}

\bibitem{IK1}
Kouyemon Iriye and Daisuke Kishimoto, \emph{Fat wedge filtrations and
  decomposition of polyhedral products}, preprint, arxiv:1412.4866.

\bibitem{arXiv:1306.6221}
\bysame, \emph{Topology of polyhedral products and the golod property of
  stanley-reisner rings}.

\bibitem{MR1793722}
John McCleary, \emph{A user's guide to spectral sequences}, second ed.,
  Cambridge Studies in Advanced Mathematics, vol.~58, Cambridge University
  Press, Cambridge, 2001. \MR{1793722 (2002c:55027)}

\bibitem{MR684544}
James Stasheff, \emph{Rational {P}oincar\'e duality spaces}, Illinois J. Math.
  \textbf{27} (1983), no.~1, 104--109. \MR{684544 (85c:55012)}

\end{thebibliography}

\end{document}